\documentclass[11pt,a4paper]{amsart}
\usepackage[utf8x]{inputenc} 
\usepackage{latexsym}
\usepackage{color,graphicx,shortvrb}
\usepackage{amsmath, amssymb}
\usepackage{amsfonts}
\usepackage[colorlinks, bookmarks=true]{hyperref}
\newtheorem{theorem}{Theorem}[section]
\newtheorem{lemma}[theorem]{Lemma}
\newtheorem{proposition}[theorem]{Proposition}
\newtheorem{corollary}[theorem]{Corollary}
\newtheorem{example}[theorem]{Example}
\theoremstyle{definition}

\newtheorem{remark}[theorem]{Remark}

\DeclareMathOperator{\GL}{GL}
\DeclareMathOperator{\fdeg}{fdeg}
\DeclareMathOperator{\fsemideg}{fsemideg}
\DeclareMathOperator{\SL}{SL}
\DeclareMathOperator{\SP}{Sp}
\DeclareMathOperator{\diag}{diag}
\title{Characterization of polynomials by their invariance properties}
\author[J. M. Almira]{Jose María Almira}
\address{J. M. Almira, Depto. Ingenier\'{\i}a y Tecnolog\'{\i}a de Computadores, Universidad de Murcia, 30100 Murcia, SPAIN}
\email{jmalmira@um.es}
\author[Y.-Q. Hu]{Ya-Qing Hu}
\address{Ya-Qing Hu, School of Mathematics, Shandong University, Jinan 250100, Shandong, P.R. CHINA}
\email{yaqinghu@sdu.edu.cn}
\begin{document}
\begin{abstract}
We prove that certain classical groups $G\subseteq \GL(d,\mathbb{R}^d)$ serve to characterize ordinary polynomials in $d$ real variables as elements of finite-dimensional subspaces of $C(\mathbb{R}^d)$ that are invariant by changes of variables induced by translations and elements of $G$. 
We also show that, if the field $\mathbb{K}$ has characteristic $0$, the elements of $\mathbb{K}[x_1,\cdots,x_d]$ admit a similar characterization for $G=\GL(d,\mathbb{K})$.   \par\vspace{0.5cm}
\hfill In memory of Prof. János Aczel, the master who inspired us all.
\end{abstract}
\maketitle
\markboth{J. M. Almira, Y.-Q. Hu}{Characterization of Polynomials}
\section{Motivation}
Given a commutative semigroup $(S,+)$ and a commutative group $(H,+)$, polynomial functions $f:S\to H$ of functional degree $\fdeg(f)\leq n$ are solutions of Fréchet's mixed functional equation: 
\begin{equation}
\label{Fmixed}
\Delta_{h_1}\Delta_{h_2}\cdots \Delta_{h_{n+1}}f=(\tau_{h_1}-\tau_0)(\tau_{h_2}-\tau_0)\cdots(\tau_{h_{n+1}}-\tau_0)f=0, \quad\forall  h_1,\cdots,h_{n+1}\in S, 
\end{equation}
and semipolynomial functions $f:S\to H$ of functional degree $\fsemideg(f)\leq n$ are solutions of Fréchet's unmixed functional equation: 
\begin{equation}
\label{Funmixed}\Delta_{h}^{n+1}f=(\tau_h-\tau_0)^{n+1}f=0, \quad\forall h\in S.
\end{equation}
Here, $\tau_h$ (resp. $\Delta_h$) is the translation (resp. forward finite difference) operator defined as $\tau_hf(x)=f(x+h)$ (resp. $\Delta_hf(x)=\tau_hf(x)-\tau_0f(x)$), 
and $\Delta_{h_1}\cdots\Delta_{h_{n+1}}f$ and $\Delta_h^{n+1}f=\Delta_h(\Delta_h^n f)$, $n=1,2,\cdots$, results from composition. 
Finally, if $H=\mathbb{K}$ is a field, then the quasipolynomial functions $f:S\to H$ are functions $f:S\to H$ satisfying the following condition: 
\begin{equation}
\label{quasi}
\dim_H {\rm span}_H\{\tau_hf: h\in S\}<\infty.
\end{equation}
All of these definitions make sense for the Schwartz distributions. 
Indeed, for $f\in \mathcal{D}(\mathbb{R}^d)'$, 
the space of real-valued Schwartz distributions defined on $\mathbb{R}^d$, 
we can introduce the translation and forward finite difference operators in the following natural way: 
\[
\tau_h (f)\{\phi\} =f\{\tau_{-h}(\phi)\}
\quad\text{and}\quad
\Delta_h(f)\{\phi\} =f\{\Delta_{-h}(\phi)\}, 
\]
where $h\in\mathbb{R}^d$ and $\phi\in\mathcal{D}(\mathbb{R}^d)$ is any test function, 
and study the distributions satisfying equations \eqref{Fmixed} and \eqref{Funmixed}, respectively. 
\par
There is a big tradition in studying polynomial, semipolynomial, and quasipolynomial functions. 
In particular, when $S=\mathbb{R}^d$ and $H=\mathbb{R}$, the relationship between these functions and ordinary polynomials $p\in\mathbb{R}[x_1,\cdots,x_d]$ is well known.  
For example, several versions of a classical theorem by Djokovich \cite{D} state that Fréchet's mixed and unmixed functional equations \eqref{Fmixed} and \eqref{Funmixed} are equivalent under very mild hypotheses \cite{La}, and, in particular, this holds when $H=\mathbb{R}$. 
Moreover, any continuous solution $f\in C(\mathbb{R}^d)$ of these equations is an ordinary polynomial \cite{F,F_1929}, 
and the same holds for the real-valued distributions defined on $\mathbb{R}^d$ \cite{AA,A_2023}. 
In addition, Anselone-Korevaar's theorem \cite{AK} states that quasipolynomial functions $f:\mathbb{R}^d\to\mathbb{R}$ are exponential polynomials. 
In other words, they are finite sums of the form
\[
f(x)=\sum_{k=0}^sp_k(x)e^{\langle \lambda_k,x\rangle}
\]
for certain ordinary polynomials $p_k(x)\in\mathbb{R}[x_1,\cdots,x_d]$ and vectors $\lambda_k\in\mathbb{R}^d$,
where $\langle x,y\rangle=\sum_{i=1}^dx_iy_i$ is the usual inner product of $\mathbb{R}^d$. 
\par
Moreover, as an extension of \cite[Lemma 15]{La}, the following result holds (the proof is included for the sake of completeness):
\begin{theorem}[Montel theorem for mixed differences]
\label{Montel_commt_mixed} 
Assume that 
either \emph{(a)} $f: G\to H$ is a map between two commutative groups and $E\subseteq G$ is a generating set for $G$, 
or \emph{(b)} $f:G\to H$ is a continuous map between two topological groups and $E$ generates a dense subgroup of $G$. 
If 
\begin{equation}
\label{montcomm}
\Delta_{h_1}\cdots\Delta_{h_{n}}f=0, \quad\forall h_1,\cdots,h_{n}\in E,
\end{equation}
then $f$ is a polynomial map of functional degree at most $n-1$.
\par
Moreover, if $f\in \mathcal{D}(\mathbb{R}^d)'$, $E$ generates a dense subgroup of $\mathbb{R}^d$, and $f$ satisfies \eqref{montcomm}, then $f$ is an ordinary polynomial of total degree $\leq n-1$ in distributional sense.  
\end{theorem}
\begin{remark}
In all that follows, we say that a distribution $f\in \mathcal{D}(\mathbb{R}^d)'$ is \textit{an ordinary polynomial} if it is an ordinary polynomial \textit{in distributional sense}. 
In particular, $f$ is equal to an ordinary polynomial in $d$ real variables almost everywhere, and henceforth, if $f\in C(\mathbb{R}^d)$, the equality holds everywhere.
\end{remark}
\begin{proof}
We proceed by induction on $n$. 
For $n=1$, we consider $\Gamma=\{h\in G: \Delta_hf=0\}$. 
Then $E\subset \Gamma$ is our hypothesis \eqref{montcomm}. 
$(a)$ Assume that $E$ generates $G$. 
For any $h,h_1,h_2\in G$, we have 
\[
\Delta_{h_1+h_2}f=\Delta_{h_1}\Delta_{h_2}f+\Delta_{h_1}f+\Delta_{h_2}f
\quad\text{and}\quad
\Delta_{-h}f=-\tau_{-h}\Delta_hf. 
\]
Therefore, $h,h_1,h_2\in \Gamma$ implies $-h,h_1+h_2\in \Gamma$ and $\Gamma$ is a subgroup of $G$ that contains $E$. 
Hence, $\Gamma=G$ and $f$ is a polynomial map of functional degree $\leq 0$. 
Assume now that the result holds for $n-1$ and take $f: G\to H$ that satisfies \eqref{montcomm}. 
Fix $h_n\in E$ and set $g=\Delta_{h_n}f$. 
Then, 
\[
\Delta_{h_1}\cdots \Delta_{h_{n-1}}g=\Delta_{h_1}\cdots \Delta_{h_{n-1}}\Delta_{h_n}f=0, \quad\forall h_1,\cdots,h_{n-1}\in E,
\]
and the induction hypothesis implies that  
\[
\Delta_{h_1}\cdots \Delta_{h_{n-1}}g=0, \quad\forall h_1,\cdots,h_{n-1}\in G. 
\]
Since the finite difference operators are commutative, we have 
\[
\Delta_{h_n}(\Delta_{h_1}\cdots \Delta_{h_{n-1}}f)=\Delta_{h_1}\cdots \Delta_{h_{n-1}}\Delta_{h_n}f=0, \quad\forall h_n\in E, \quad\forall h_1,\cdots,h_{n-1}\in G.
\]
Since the result holds for $n=1$, we conclude that the above equation holds for all $h_n\in G$.
This proves the result according to the hypotheses in (a). 
If we assume (b), the above computation shows that equation \eqref{montcomm} is satisfied for all steps $h_1,\cdots,h_n$ in a dense subgroup of $G$, and the continuity of $f$ implies that the equation is satisfied for all steps $h_1,\cdots,h_n\in G$. 
Hence, $f$ is a continuous polynomial function of functional degree at most $n-1$. 
\par
The very same arguments apply when $f$ is a real-valued distribution defined on $\mathbb{R}^d$, 
since the operators $\tau_h$ and $\Delta_h=\tau_h-\tau_0$ can be naturally extended to $\mathcal{D}(\mathbb{R}^d)'$ and they inherit the properties of the original operators defined on the space of test functions. 
\end{proof}
In this paper, we investigate polynomials from a different perspective. 
Concretely, we prove that ordinary polynomials defined on $\mathbb{R}^d$ can be characterized by their invariance properties with respect to changes of variables induced by certain affine transformations. 
This topic appeared a long time ago in the research of Loewner \cite{Lo} and, later, of Laird and McCan \cite{LMc}. 
Our contribution is a natural extension of the ideas introduced by the first author in \cite{A}.  
We also prove that ordinary polynomials defined on $\mathbb{K}^d$, 
where $\mathbb{K}$ is any field of characteristic $0$, 
can be characterized by their invariance properties with respect to translations and linear isomorphisms of $\mathbb{K}^d$. 
\section{A general technique}
Let $G$ be a subgroup of $\GL(d,\mathbb{R})$, which acts naturally on $\mathbb{R}^d$ from the left. 
For any $P\in G$ and $h\in\mathbb{R}^d$, we consider operators $O_P,\tau_h:C(\mathbb{R}^d)\to C(\mathbb{R}^d)$ given by 
\[
O_P(g)(x)=g\circ P(x)=g(Px)\quad\text{and}\quad\tau_h(g)(x)=g(x+h)\quad\text{for}\quad g\in C(\mathbb{R}^d).
\]
Moreover, for the distributions $f\in \mathcal{D}(\mathbb{R}^d)'$, we extend the above operators in a natural way:
\[
O_P(f)\{\phi\}=\frac{1}{|\det(P)|}f\{O_{P^{-1}}(\phi)\}\quad\text{and}\quad\tau_h(f)\{\phi\}=f\{\tau_{-h}\phi\}
\] 
where $P\in G$ is any invertible matrix, 
$\phi\in\mathcal{D}(\mathbb{R}^d)$ is any test function,  
$O_{P^{-1}}(\phi)(x)=\phi(P^{-1}x)$, and 
$\tau_{-h}(\phi)(x)=\phi(x-h)$, for all $x\in\mathbb{R}^d$.
\par
Let $X_d$ denote indistinctly $C(\mathbb{R}^d)$ or $\mathcal{D}(\mathbb{R}^d)'$. 
Let $V$ be a vector subspace of $X_d$. 
We say that $V$ is translation invariant if $\tau_h(V)\subseteq V$ for all $h\in\mathbb{R}^d$. 
Moreover, given any subgroup $G\subseteq \GL(d,\mathbb{R})$, we say that $V$ is invariant by elements of $G$ if $O_P(V)\subseteq V$ for all $P\in G$. 
\par
For each subgroup $G\subseteq\GL(d,\mathbb{R})$ and $z_0\in\mathbb{R}^d$, the orbit of $z_0$ in $\mathbb{R}^d$ under the action of $G$ on $\mathbb{R}^d$ is denoted by $Gz_0=\{Pz_0: P\in G\}$ and the set 
\[
\Lambda_{G,z_0}=Gz_0-Gz_0=\{Pz_0-Qz_0: P,Q\in G\}\subseteq \mathbb{R}^d
\]
is defined as the Minkowski difference of $Gz_0$ with itself. 
The group $G$ and the point $z_0$ will be called, respectively, the starting group and the starting point of the set $\Lambda_{G,z_0}$. 
\par
Of course, the same notation can be used if the functions $f$ are defined in $\mathbb{K}^d$ and take values on $\mathbb{K}$ for any field $\mathbb{K}$. 
In such a case, we would take $G\subseteq\GL(d,\mathbb{K})$. 
\par
In this section, we prove the following:
\begin{theorem}
\label{mainT}
Let $G$ be a subgroup of $\GL(d,\mathbb{R})$ and $V$ be a finite-dimensional vector subspace of $X_d$.
If $\Lambda_{G,z_0}$ has a nonempty interior for certain $z_0\in\mathbb{R}^d$, $\tau_h(V)\subseteq V$ for all $h\in\mathbb{R}^d$, and $O_P(V)\subseteq V$ for all $P\in G$, 
then all elements of $V$ are ordinary polynomials in $d$ real variables. 
\end{theorem}
\begin{remark}
\label{primera}
It is interesting to know for which $f\in X_d$ we have $\dim R_G(f)=N<\infty$, 
where $R_G(f)$ denotes the smallest subspace of $X_d$ that contains $f$ and is invariant by operators $\tau_h$ and $O_P$ for all $h\in\mathbb{R}^d$ and $P\in G$. 
If $\Lambda_{G,z_0}$ has a nonempty interior, Theorem \ref{mainT} proves that $f$ is an ordinary polynomial in $d$ real variables. 
The functional degree of $f$ (which is the total degree of $f$ as an ordinary polynomial) can be bounded in terms of the dimension of $R_G(f)$. 
Indeed, if $f$ is a nonzero constant, then $R_G(f)$ has dimension $1$; 
If $f$ is an ordinary polynomial of total degree exactly $n\ge1$,  
we can choose the steps $h_1,h_2,\cdots,h_n\in\mathbb{R}^d$ such that the functions $f,\Delta_{h_1}f, \Delta_{h_2}\Delta_{h_1}f,\cdots,\Delta_{h_n}\Delta_{h_{n-1}}\cdots\Delta_{h_1}f$ are linearly independent 
(this is because, for every ordinary polynomial $g$, and every nonzero step $h$, $\Delta_h g$ strictly reduces the total degree of $g$, and $\Delta_{h_n}\Delta_{h_{n-1}}\cdots \Delta_{h_1}f$ is a nonzero constant for a certain choice of the steps $h_1,\cdots,h_n$,
since the total degree and the functional degree of $f$ coincide). 
Moreover, all these functions belong to $R_G(f)$, which means that $n\leq \dim R_G(f)-1$. 
\par
In certain cases, the relationship between $n$ and $\dim R_G(f)$ can be strengthened, and what usually happens is that $\dim R_G(f)$ is much larger than $n$. 
Indeed, if we denote by $x^a=\prod_{i=1}^dx_i^{a_i}$ the monomial of total degree $|a|=\sum_{i=1}^n a_i$, 
where $x=(x_1,\ldots,x_d)$ and $a=(a_1,\ldots,a_d)\in\mathbb{N}^d$, 
then $f(x)=\sum_{|a|\leq n}c_{a}x^a\in\mathbb{R}[x_1,\ldots,x_d]$ is an ordinary polynomial in $d$ real variables of total degree at most $n$. 
It is clear that $R_G(f)$ is a subspace of the vector space of ordinary polynomials in $d$ real variables of total degree at most $n$, and henceforth its dimension is bounded by that of the latter, i.e., $\dim R_G(f)\leq \binom{n+d}{d}$ for all $G$, and sometimes two quantities coincide. 
Thus: 
\[
\dim R_G(f)\le\binom{n+d}{d}=\frac{(n+d)!}{n!d!}\le(n+1)^d. 
\]
implies that $n\ge\sqrt[d]{\dim R_G(f)}-1$. 
On the other hand, in case that $\dim R_G(f)=\binom{n+d}{d}$, we get $\binom{n+d}{d}\ge\frac{(n+1)^d}{d!}$, 
which implies that $(n+1)^d\leq d!\dim R_G(f)$ and, henceforth,
\[
n\le\sqrt[d]{d!\dim R_G(f)}-1
\]
-- a much better bound than $n\leq \dim R_G(f)-1$.
\end{remark}
\begin{lemma}[Generalization of Fr\'{e}chet's Theorem]
\label{pre}
Let $f\in X_d$. 
Assume that $q(\tau_y)f=0$ for all $y$ with $\|y-z_0\|\leq \delta$, for certain $z_0\in\mathbb{R}^d$, $\delta>0$, and  polynomial $q(z)=a_0+a_1z+\cdots a_nz^n$ with $a_0\neq 0$. 
Then $f$ is an ordinary polynomial in $d$ real variables of total degree at most $n-1$. 
\end{lemma}
\begin{proof}
By assumption, $q(\tau_y)f=0$ for all $y$ with $\|y-z_0\|\leq \delta$, which means that 
\begin{equation*}
0=\sum_{k=0}^na_k(\tau_{ky}f)(x),
\quad\forall y\in B_{d}(z_0,\delta):=\{z\in\mathbb{R}^d:\|z-z_0\|\leq \delta\}.
\end{equation*}
Assume that $y\in B_{d}(z_0,\delta/2),h_1\in  B_{d}(0,\delta/2)$ (so that $y^*=y-h_1\in B_d(z_0,\delta)$). 
Then:
\begin{equation*} 
0=\sum_{k=0}^na_k\tau_{ky^*}f(x)
 =\sum_{k=0}^na_k\tau_{ky-kh_1}f(x),
\quad\forall y\in B_{d}(z_0,\delta/2),\forall h_1\in B_{d}(0,\delta/2).
\end{equation*}
Apply $\tau_{nh_1}$ to both sides of the equation. 
Then, for all $y\in B_{d}(z_0,\delta/2)$ and $h_1\in B_{d}(0,\delta/2)$, 
\begin{equation*}
0=\sum_{k=0}^na_k\tau_{nh_1}\tau_{ky-kh_1}f(x)
 =\sum_{k=0}^{n}a_k\tau_{(n-k)h_1}(\tau_{ky}f)(x).
\end{equation*}
Taking differences with $0=\sum_{k=0}^na_k(\tau_{ky}f)(x)$, we obtain  
\begin{equation*}
0=\sum_{k=0}^{n-1}a_k\Delta_{(n-k)h_1}(\tau_{ky}f)(x) 
 =\sum_{k=0}^{n-1}a_k\tau_{ky} (\Delta_{(n-k)h_1} f)(x),
\quad\forall y\in B_{d}(z_0,\delta/2),\forall h_1\in B_{d}(0,\delta/2).
\end{equation*}
We can repeat the argument, reducing the norm of $y-z_0,h_1$ to $\delta/4$, and, then, to $\delta/8$, etc., which leads, after $n$ iterations of the argument, to the equation
\[
a_0\Delta_{h_n} \Delta_{2h_{n-1}} \cdots \Delta_{(n-1)h_2} \Delta_{nh_1}(f)(x)=0,\quad\forall h_1,\cdots,h_n\in B_{d}(0,\delta/2^n).
\] 
In particular, we have 
\[
\Delta_{h_n} \Delta_{h_{n-1}} \cdots \Delta_{h_2} \Delta_{h_1}(f)(x) 
=\Delta_{h_n} \Delta_{2\frac{h_{n-1}}{2}} \cdots \Delta_{(n-1)\frac{h_2}{n-1}} \Delta_{n\frac{h_1}{n}}(f)(x) 
=0
\] 
whenever $h_1,\cdots,h_n\in B_{d}(0,\delta/2^n)$, 
since $\|h\|\leq \delta/2^n$ implies $\|h/k\|\leq \delta/2^n$ for $k=1,2,\cdots,n$. 
\par
The result follows from Theorem \ref{Montel_commt_mixed} since every Euclidean ball $B_{d}(0,\varepsilon)$ of positive radius contains a generating set of a dense subgroup of $\mathbb{R}^d$ (a simple way to construct such a set is taking into account Kronecker's theorem \cite[Theorem 442, page 382]{HW}, 
which implies that the set $E=\{e_1,\cdots,e_d,(\theta_1,\cdots,\theta_d)\}$ generates a dense subgroup of $\mathbb{R}^d$ as long as $\{1,\theta_1,\cdots,\theta_d\}$ is a $\mathbb{Q}$-linear independent set of real numbers. 
Obviously, the same property will hold for the sets $E_n=\frac{1}{n}E$, $n\in\mathbb{N}$, and clearly $E_n\subseteq B_d(0,\varepsilon)$ if $n$ is large enough).  
\end{proof}
We also use the following well-known result \cite{Fro,RR}: 
\begin{theorem}[Frobenius, 1896]
Let $\mathbb{K}=\overline{\mathbb{K}}$ be an algebraically closed field, and $V$ be a finite-dimensional $\mathbb{K}$-vector space. 
Assume that $T,S:V\to V$ are commuting linear operators (i.e., $TS=ST$). 
Then they are simultaneously triangularizable on $V$. 
\end{theorem}
\begin{corollary}
\label{coFro}
Let $\mathbb{K}$ be any field and $V$ be a $\mathbb{K}$-vector space of finite dimension $n=\dim_{\mathbb{K}}(V)$.  
Assume that $T,S:V\to V$ are commuting linear operators. 
Then the eigenvalues $\lambda_i(T), \lambda_i(S)$ and $\lambda_i(TS)$ of $T,S$ and $TS$ (which belong to the algebraic closure $\overline{\mathbb{K}}$ of $\mathbb{K}$), respectively, can be arranged, counting multiplicities, so that $\lambda_i(TS)=\lambda_i(T)\lambda_i(S)$ for $i=1,\cdots,n$.  
In particular, the characteristic polynomials 
$p(z)=\prod_{i=1}^{n}(z-\lambda_i(T))$, 
$q(z)=\prod_{i=1}^{n}(z-\lambda_i(S))$, and $h(z)=\prod_{i=1}^{n}(z-\lambda_i(T)\lambda_i(S))$, 
of $T,S$, and $TS$, respectively, are polynomials in $\mathbb{K}[z]$. 
\end{corollary}
\begin{proof}
Let $n=\dim_{\mathbb{K}}(V)$ and fix a basis of $V$. 
Consider the matrices $A,B$ associated to $T,S$, respectively, with respect to this basis. 
Then $TS=ST$ implies that $AB=BA$, and we consider the operators $\overline{T},\overline{S}:\overline{\mathbb{K}}^n\to\overline{\mathbb{K}}^n$ given by $\overline{T}(x)=Ax$, $\overline{S}(x)=Bx$. 
We apply the Frobenius theorem to these operators to claim that the matrices $A,B$ are simultaneously triangularizable on $\overline{\mathbb{K}}$. 
In particular, the eigenvalues $\lambda_i(A),\lambda_i(B),\lambda_i(AB)$ of $A,B,AB$ can be arranged, counting multiplicities, in such a way that 
$\lambda_i(AB)=\lambda_i(A)\lambda_i(B)$ for $i=1,\cdots,n$. This ends the proof. 
\end{proof}
Knowing for what $G\subseteq\GL(d,\mathbb{R})$ and $z_0\in\mathbb{R}^d$, the set $\Lambda_{G,z_0}=Gz_0-Gz_0$ has a nonempty interior is crucial in our argument. 
We summarize some basic results in the following lemma. 
\begin{lemma}
\label{lemma_general}
Let $G$ be a subgroup of $\GL(d,\mathbb{R})$ and $z_0$ be a point in $\mathbb{R}^d$. 
\begin{enumerate} 
\item If the set $\Lambda_{G,z_0}$ has a nonempty interior, then $z_0\neq0$.
\item If $-I_d\in G$, where $I_d$ is the identity matrix, then $\Lambda_{G,z_0}$ is also the Minkowski sum $Gz_0+Gz_0$ of the orbit $Gz_0$ with itself. 
\item The set $\Lambda_{G,z_0}$ is invariant by replacing  $z_0$ by any other point in the orbit $Gz_0$. 
\item If $G=\GL(d,\mathbb{R})$ and $z_0\ne0$, then $Gz_0=\mathbb{R}^d\setminus\{0\}$ and $\Lambda_{G,z_0}=\mathbb{R}^d$. 
\item The set $\Lambda_{G,z_0}$ can be written as 
\begin{equation}
\label{decomposition}
\Lambda_{G,z_0}=G(Gz_0-z_0)=\{P(Q-I)z_0:P,Q\in G\}.
\end{equation}
\end{enumerate}
\end{lemma}
\begin{proof}
(i) We have $G0=\{0\}$ and thus $\Lambda_{G,0}=\{0\}$ for all $G$.
\par
(ii) We have $-Gz_0=G(-z_0)=G(-I_d)z_0=Gz_0$. 
\par
(iii) If $z_1\in Gz_0$, that is, $z_1=Pz_0$ for some $P\in G$, then $Gz_1=GPz_0=Gz_0$ and thus 
\[
\Lambda_{G,z_1}=Gz_1-Gz_1=Gz_0-Gz_0=\Lambda_{G,z_0}. 
\]
\par
(iv) If $z_0\ne0$ and $z_1\ne0$, then there is a $P\in\GL(d,\mathbb{R})$ such that $z_1=Pz_0$. 
(This follows from a simple fact in linear algebra: 
if $\{u_i\}_{i\in I}$ is a basis of a vector space $U$ and $\{v_i\}_{i\in I}$ is any set of vectors in a vector space $V$, then there is a unique linear map $f:U\to V$ such that $v_i=f(u_i)$ for all $i\in I$.) 
Hence, we have $Gz_0=\mathbb{R}^d\setminus\{0\}$ and $\Lambda_{G,z_0}=\mathbb{R}^d$. 
\par
(v) This is because the multiplication map $G\times G\to G;\ (P,Q)\mapsto PQ$ is subjective. 
\end{proof}
One might guess that $\Lambda_{G,z}$ has a nonempty interior for all $z\ne0$ if it has a nonempty interior for some $z\ne0$.
However, this is not always true, unless that $G$ is a normal subgroup of $\GL(d,\mathbb{R})$. 
\begin{lemma}
\label{normal} 
If $G$ is a normal subgroup of $\GL(d,\mathbb{R})$, then 
$\Lambda_{G,z}$ has a nonempty interior for all $z\ne0$, provided that $\Lambda_{G,z_0}$ has a nonempty interior for some $z_0\ne0$.
\end{lemma}
\begin{proof}
For any $z_1\neq0$, let $P\in\GL(d,\mathbb{R})$ be such that $z_1=Pz_0$. 
Since $G$ is a normal subgroup of $\GL(d,\mathbb{R})$, we have $Gz_1=GPz_0=PP^{-1}GPz_0=PGz_0$ and thus 
\[
\Lambda_{G,z_1}=Gz_1-Gz_1=P(Gz_0-Gz_0)=P\Lambda_{G,z_0}. 
\]
Since multiplication by an invertible matrix $P$ is an open map on $\mathbb{R}^d$, the image of the nonempty interior of $\Lambda_{G,z_0}$ is open in $\mathbb{R}^d$. 
Therefore, $\Lambda_{G,z_1}$ also has a nonempty interior. 
\end{proof}
\begin{remark}
The normality condition on $G$ in Lemma \ref{normal} cannot be removed. 
For example, we consider the group of diagonal matrices in dimension $2$: 
\[
D_2=\left\{\diag(a,b)= \begin{pmatrix} a & 0\\ 0 & b \end{pmatrix} : ab\ne0\right\},
\]
which is not a normal subgroup in $\GL(2,\mathbb{R})$, since 
\[
\begin{pmatrix}1 & 1\\ 0 & 1 \end{pmatrix}
\begin{pmatrix}1 & 0\\ 0 & 2 \end{pmatrix}
\begin{pmatrix}1 & 1\\ 0 & 1 \end{pmatrix}^{-1}=
\begin{pmatrix}1 & 1\\ 0 & 2 \end{pmatrix}\not\in D_2.
\]
Let $z_0=(1,1)^T$ and $z_1=(1,0)^T$ be two nonzero vectors. 
Then, we have 
\[
D_2z_0=\{(a,b)^T:ab\ne0\}\quad\text{and}\quad D_2z_1=\{(a,0)^T:a\ne0\}. 
\]
Hence, $\Lambda_{D_2,z_0}=\mathbb{R}^2$ has a nonempty interior, but $\Lambda_{D_2,z_1}=\mathbb{R}\times\{0\}$ has an empty interior. 
\end{remark}
\begin{remark}
The orbit-stabilizer theorem explains such a phenomenon. 
For any subgroup $G\subseteq\GL(d,\mathbb{R})$ and a fixed point $z\in\mathbb{R}^d$, we consider the map $\phi:G\to\mathbb{R}^d$ by sending $P$ to $Pz$, whose image is the orbit $Gz$. 
Two elements $g$ and $h$ have the same image if and only if 
they lie in the same left coset for the stabilizer subgroup $G_z=\{g\in G: gz=z\}$. 
Thus, $\phi$ induces a bijection between the set $G/G_z$ of left cosets for the stabilizer subgroup $G_z$ and the orbit $Gz$ sending $PG_z$ to $Pz$. 
Hence, the orbit is ``smaller'', if the stabilizer is ``larger'', i.e., if the point is ``more symmetric''. 
\par
In our situation, for any $P\in\GL(d,\mathbb{R})$, if $z\ne0$, then $Pz\ne0$, but the stabilizer subgroups $G_z$ and $G_{Pz}$ may have different ``sizes'', so do the corresponding orbits $Gz$ and $GPz$. 
But if $G$ is a normal subgroup, then $GPz=PGz$, which has roughly the same size as $Gz$, since the multiplication by $P$ induces an open map on $\mathbb{R}^d$. 
So if one wants $Gz$ to have a nonempty interior, a nonzero point with ``less symmetry'' would be a better choice for the starting point $z$.
\end{remark}
The interior information of the orbit $Gz_0$ is not always the same as that of $\Lambda_{G,z_0}$, as the Minkowski sum or difference often ``fatten'' geometric objects. 
\begin{lemma} \label{fatten} Let $G$ and $z_0$ be as before. 
\begin{enumerate}
\item If the orbit $Gz_0$ has a nonempty interior, then so does $\Lambda_{G,z_0}$.
\item There exits $G$ and $z_0$ such that the orbit $Gz_0$ has an empty interior, but $\Lambda_{G,z_0}$ does not. 
In particular, this happens, for all $d\geq 2$, $z_0\ne0$, and $G=\mathbf{O}(d)$ or $G=\mathbf{SO}(d)$, the orthogonal or special orthogonal group in dimension $d$. 
\end{enumerate}
\end{lemma}
\begin{proof}
(i) Each translation $Gz_0+z$ has a nonempty interior, and so does 
\[
\Lambda_{G,z_0}=\bigcup_{z\in -Gz_0} (Gz_0+z).
\]
\par
(ii) Take $G=\mathbf{O}(d)$ or $G=\mathbf{SO}(d)$. 
For any $z_0\ne0$, let $r=\|z_0\|>0$ be its length.  
The orbit $Gz_0$ (resp. its translation $Gz_0-z_0$) is the sphere with center $0$ (resp. $-z_0$) and radius $r$ in $\mathbb{R}^d$, so it has an empty interior. 
Note that $Gz_0-z_0$ contains at least one vector whose length is any given number in the interval $[0,2r]$. 
Therefore, by \eqref{decomposition}, $\Lambda_{G,z_0}=G(Gz_0-z_0)$ contains a ball of center $0$ and radius $2r$, which has a nonempty interior. 
\end{proof}
\begin{example} 
\label{SO2}
For a concrete example, consider the special orthogonal group in dimension $2$: 
\[
G=\mathbf{SO}(2)=\left\{\begin{pmatrix}
\cos\alpha & -\sin\alpha\\ 
\sin\alpha & \sin\alpha \end{pmatrix}:\alpha\in[0,2\pi)\right\}.
\]
Take $z_0=(1,0)^T$.  
The orbit $Gz_0$ is the unit circle $S^1=\{(\cos\alpha,\sin\alpha)^T:\alpha\in [0,2\pi)\}$ in $\mathbb{R}^2$ and $\Lambda_{G,z_0}$ is the Minkowski sum of $S^1$ with itself, as $-I_2\in G$. 
For any $\alpha,\beta\in[0,2\pi)$, we have 
\[
(\cos\alpha+\cos\beta,\sin\alpha+\sin\beta)^T=2\cos\frac{\alpha-\beta}{2}\left(\cos\frac{\alpha+\beta}{2},\sin\frac{\alpha+\beta}{2}\right)^T,
\]
which parametrizes the ball with center $0$ and radius $2$ in $\mathbb{R}^2$. 
\end{example}
\begin{proof}[Proof of Theorem \ref{mainT}]
Note that the operators $\tau_y: V\to V$ are all injective and thus surjective, therefore, the eigenvalues of $\tau_y$ are all nonzero. 
They are also commutative, i.e.,  
\begin{equation} 
\label{traslation}
\tau_{y+z}=\tau_y\tau_z=\tau_z\tau_y,\quad\forall y,z\in \mathbb{R}^d,
\end{equation}
and the operators $\tau_{Py}$ and $\tau_y$ are similar to each other, i.e., 
\begin{equation} 
\label{similarity}
\tau_{Py}=(O_{P})^{-1}\tau_y O_{P}=O_{P^{-1}}\tau_y O_{P},\quad\forall P\in\GL(d,\mathbb{R}),\quad\forall y\in\mathbb{R}^d. 
\end{equation}
It follows that  
\begin{itemize}
\item[(a)] For any $P\in\GL(d,\mathbb{R})$ and  $y\in\mathbb{R}^d$, the operators $\tau_{y}:V\to V$ and $\tau_{Py}:V\to V$ have the very same eigenvalues and characteristic polynomial. 
\item[(b)] A direct application of Corollary \ref{coFro} guarantees that for any $y,z\in\mathbb{R}^d$, 
we can arrange the eigenvalues $\lambda_i(\tau_y),\lambda_i(\tau_z),\lambda_i(\tau_{y+z})$ of  $\tau_{y},\tau_z,\tau_{y+z}$, respectively,  in such a way that
\[
\lambda_i(\tau_{y+z})=\lambda_i(\tau_y)\lambda_i(\tau_z),\quad i=1,\cdots,N:=\dim V.
\]
\end{itemize}
\par
Now let $P\in G$ and $y\in\mathbb{R}^d$. 
From (a) we have that, for a certain permutation $\sigma$ of $\{1,\cdots, N\}$ 
(which depends on $P\in G$ and the enumeration of the eigenvalues of $\tau_y$),  
\[
\lambda_i(\tau_{Py})=\lambda_{\sigma(i)}(\tau_y),\quad i=1,\cdots,N.
\]
Moreover, if we set $z=(P-I)y$, then $Py=y+(P-I)y$ and $(b)$ imply that:
\[
\lambda_{\sigma(i)}(\tau_y)=\lambda_i(\tau_{Py})=\lambda_i(\tau_y)\lambda_i(\tau_{(P-I)y}),\quad i=1,\cdots,N.
\]
Hence, we have 
\begin{equation}
\label{autova}
\lambda_i(\tau_{(P-I)y})=\frac{\lambda_{\sigma(i)}(\tau_y)}{\lambda_i(\tau_y)}, \quad i=1,\cdots,N. 
\end{equation}
\par
Let $z_0\in\mathbb{R}^d$ be such that $\Lambda_{G,z_0}$ has a nonempty interior. 
Fix the enumeration of the eigenvalues of $\tau_{z_0}$. 
Consider the polynomial 
\begin{equation}
\label{pol_q}
q(z)=\prod_{\sigma\in S_N^*}\prod_{i=1}^N(z- \frac{\lambda_{\sigma(i)}(\tau_{z_0})}{\lambda_i(\tau_{z_0})}),
\end{equation}
where $S_N^*$ denotes the set of permutations of $\{1,\cdots, N\}$ that 
appear in \eqref{autova} for some $P\in G$.  
Then $\prod_{i=1}^N(z-\frac{\lambda_{\sigma(i)}(\tau_{z_0})}{\lambda_i(\tau_{z_0})})$ has all its coefficients in $\mathbb{R}$ for each permutation $\sigma\in S_N^*$ and consequently all the coefficients of the polynomial $q$ defined by \eqref{pol_q} are real numbers. 
(Obviously, it may happen that $S_N^*=S_N$.)
\par
Thus $q(z)=a_0+a_1z+\cdots a_mz^m$ for certain coefficients $a_k\in\mathbb{R}$, with $m\leq N\cdot N!$, and $a_0\neq 0$ because all eigenvalues $\lambda_i(\tau_{z_0})$ are different from $0$ (since $\tau_{z_0}$ is injective). 
Moreover, $q(z)$  is a multiple of the characteristic polynomial of the operator $\tau_{(P-I)z_0}:V\to V$, for every $P\in G$, 
and the same holds for the operators $\tau_{Q(P-I)z_0}:V\to V$, for every $P,Q\in G$, 
since $\tau_{Qw}:V\to V$ and $\tau_{w}:V\to V$ have the same characteristic polynomial for all $Q\in G$ and all $w\in\mathbb{R}^d$. 
Thus, equation \eqref{decomposition} implies that $q(z)$ is a multiple of the characteristic polynomial of $\tau_z:V\to V$ for every $z\in\Lambda_{G,z_0}$. 
Now Cayley-Hamilton's theorem implies that $q(\tau_{z})=0$ for all $z\in\Lambda_{G,z_0}$, so that we can apply Lemma \ref{pre} to conclude that $f$ is, in distributional sense, an ordinary polynomial of total degree at most $m-1$. 
In particular, if $f\in C(\mathbb{R}^d)$, then $f$ is an ordinary polynomial in $d$ real variables of total degree at most $m-1$.    
\end{proof}
\begin{remark}
Assume that $V\subseteq X_d$ satisfies the hypotheses of Theorem \ref{mainT} and $f\in V$. 
Then $R_G(f)\subseteq V$ so that we can apply Remark \ref{primera} to get a better estimation:
\[
\sqrt[d]{\dim R_G(f)}-1\leq {\rm fdeg}(f)\leq \dim R_G(f)-1\leq \dim V-1.
\] 
\end{remark}
\par
Clearly, item (ii) of Lemma \ref{fatten} implies that Theorem \ref{mainT} can be applied to $G=\mathbf{SO}(d)$ and $G=\mathbf{O}(d)$ for any $d\geq 2$. 
Concretely, the following theorem holds true:
\begin{theorem}
\label{On}
Assume that $d\geq 2$ and $G=\mathbf{SO}(d)$ or $G=\mathbf{O}(d)$. 
If $f$ belongs to a finite-dimensional vector space $V\subset X_d$ such that $V$ is invariant by translations and elements of $G$,
then $f$ is an ordinary polynomial in $d$ real variables. 
\end{theorem}
We have demonstrated that if $G$ is large enough 
then the finite-dimensional subspaces of $X_d$ which are invariant by translations and elements of $G$ are formed by ordinary polynomials. 
Now we study the opposite extreme. 
Concretely,  we solve the case when $G$ is finite. 
\par
If $G=\{I_d\}$, then everything reduces to Anselone-Korevaar's theorem \cite{AK}, which claims that the elements of any finite-dimensional vector space $V\subseteq X_d$ that is invariant by translations are exponential polynomials. 
Now, we can prove the following:
\begin{theorem}
\label{finite}
Let $V\subseteq X_d$ be a vector space of finite dimension and invariant by translations, and let $G\subseteq \GL(d,\mathbb{R})$ be any finite subgroup. 
Then there exists another finite-dimensional vector space $W\subseteq X_d$ such that $V\subseteq W$ and $W$ is invariant by translations and elements of $G$. 
\end{theorem}
\begin{proof}
Let $P_1,\cdots,P_s$ be an enumeration of elements of $G$, where $P_1=I_d\in \GL(d,\mathbb{R})$. 
Set $\underline{s}=\{1,\cdots,s\}$. 
Let $\eta:\underline{s}\times\underline{s}\to\underline{s}$ be a map such that $P_iP_j=P_{\eta(i,j)}$. 
Thus, $\eta$ describes the group law of $G$. 
If $\{g_k: k=1,\ldots,\dim V\}$ is a basis of $V$, 
then 
\[
W={\rm span}\{O_{P_j}(g_k): j\in\underline{s},\ k=1,\ldots,\dim V\}
\]
is a finite-dimensional vector space $W\subseteq X_d$ containing $V=O_{P_1}(V)$. 
Moreover, for any $i,j\in\underline{s}$, $k=1,\ldots,\dim V$, and $y\in\mathbb{R}^d$, 
the following identities 
\[
O_{P_{i}}(O_{P_{j}}(g_{k}))=O_{P_{j}P_{i}}(g_{k})=O_{P_{\eta(j,i)}}(g_{k})\in W
\quad\text{and}\quad 
\tau_y(O_{P_j}(g_k))=O_{P_j}(\tau_{P_jy}(g_k))\in W
\]
imply that $W$ is invariant by translations and elements of $G$, since $\tau_{P_jy}(g_k)\in V$. 
\end{proof}
\begin{remark}
As we have already observed, for each $d\geq 2$, a distribution $f\in X_d$ is an ordinary polynomial if and only if it belongs to a finite-dimensional space of functions $V\subset X_d$ invariant by isometries of $(\mathbb{R}^d,\|\cdot\|_2)$.  
On the other hand, a famous Theorem by Ulam and Mazur (see \cite{UM,V}) claims that, if $(X,\|\cdot\|)$, $(Y,\|\cdot\|)$ are normed spaces and $T:X\to Y$ is a surjective isometry, then $T$ must be an affine map. 
In particular, for every given norm $\|\cdot\|$ on $\mathbb{R}^d$, the isometries $P:(\mathbb{R}^d,\|\cdot\|)\to (\mathbb{R}^d,\|\cdot\|)$ are affine maps, and we may wonder which elements $f\in X_d$ belong to the finite-dimensional spaces $V\subset X_d$ invariant by these isometries. 
Since these isometries are affine maps, every ordinary polynomial has this property, and every distribution with this property is necessarily an exponential polynomial. 
\par
Moreover, a full characterization depends on the chosen norm $\|\cdot\|$. 
For example, it is known \cite{R} that a norm $\|\cdot\|$ defined on $\mathbb{R}^d$ comes from an inner product (i.e., $\|x\|=\langle x,x\rangle^{\frac{1}{2}}$ for an inner product $\langle\cdot,\cdot\rangle$) if and only if the linear isometries $P:(\mathbb{R}^d,\|\cdot\|) \to (\mathbb{R}^d,\|\cdot\|)$ form a transitive group, which means that for each pair $x,y\in\mathbb{R}^d$ of norm-$1$ elements, $\|x\|=\|y\|=1$, there is a linear isometry $P\in \GL(d,\mathbb{R})$ such that $Px=y$. 
In particular, if we denote by $G={\rm Iso}(\mathbb{R}^d,\|\cdot\|)$ the group of linear isometries, then $\Lambda_{G,e_1}$ has a nonempty interior, and we can apply Theorem \ref{mainT} to $G$. 
This means that a result analogous to Theorem \ref{On} holds for $G={\rm Iso}(\mathbb{R}^d,\|\cdot\|)$ when the norm $\|\cdot\|$ is induced by an inner product.   
\par 
Let us consider another example: given $1\leq p\leq +\infty$, $p\neq 2$, it is known that ${\rm Iso}(\mathbb{R}^d,\|\cdot\|_p)$ is the finite subgroup of $\GL(d,\mathbb{R})$ formed by the signed permutation matrices of size $d$ \cite{LS}, which is isomorphic to the hyperoctahedral group (the group of symmetries of the $d$-dimensional hypercube as well as the group of symmetries of the $d$-dimensional cross-polytope) and has size $2^d\cdot d!$.
Consequently, Theorem \ref{finite} implies that a distribution $f\in X_d$ is an exponential polynomial if and only if it belongs to a finite-dimensional space $V\subseteq X_d$ which is invariant by isometries of $(\mathbb{R}^d,\|\cdot\|_p)$, for some (equivalently, for any) $p\in [1,\infty]\setminus\{2\}$. 
\par   
Finally, we should mention that ${\rm Iso}(\mathbb{R}^d,\|\cdot\|)$ is infinite if and only if the numerical index of $(\mathbb{R}^d,\|\cdot\|)$ is zero (see \cite{M,R}), and (for $d\geq 3$) there are finite-dimensional non-Hilbert spaces satisfying this condition (for $d=2$ the only normed space with numerical index zero is the Hilbert space) \cite{MMRP}. 
In these cases, we cannot directly use neither Theorem \ref{mainT} nor Theorem \ref{finite}, and a full characterization of the finite-dimensional spaces $V\subseteq X_d$ invariant by isometries of $(\mathbb{R}^d,\|\cdot\|)$ is still an open problem. 
\end{remark}
\par
We devote the next two sections of the paper to show a few new examples of classical groups $G\subseteq \GL(d,\mathbb{R})$ with the property that if $f$ belongs to a finite-dimensional vector space of continuous functions
$V\subset C(\mathbb{R}^d)$ such that $V$ is invariant by translations and elements of $G$,
then $f$ is an ordinary polynomial in $d$ real variables. 
In Section \ref{Dist}, we prove the results from Sections \ref{SOpq} and \ref{SSp} for distributions. 
The last section of the paper studies a similar problem for functions defined on $\mathbb{K}^d$, where $\mathbb{K}$ is any field of characteristic $0$. 
\section{The general orthogonal group $\mathbf{O}(p,q)$ case} \label{SOpq}
For the proof of the main results in this and next sections, we use the following theorem by Kiss and Laczkovich (see \cite{KL}):
\begin{theorem} 
\label{KLseparatelypol} 
Let $\mathbb{K}$ denote either the field of real numbers, $\mathbb{K}=\mathbb{R}$, 
or the field of complex numbers, $\mathbb{K}=\mathbb{R}$. 
Let $G, H$ be commutative topological groups and assume that $G$ is a connected Baire space and that $H$ contains a dense subgroup of finite torsion-free rank. 
If $f: G\times H\to\mathbb{K}$ is a separately continuous polynomial map, then $f$ is a continuous polynomial map. 
\end{theorem}
\begin{remark}
By definition, $f:G\times H\to\mathbb{K}$ is a separately continuous polynomial map if for each $(a,b)\in G\times H$ the functions $f_a:H\to\mathbb{K}; y\mapsto f(a,y)$ and $f^b:G\to\mathbb{K}; x\mapsto f(x,b)$ are continuous polynomial maps. 
In particular, this concept does not assume that $f:G\times H\to\mathbb{K}$ is continuous (i.e., the continuity of $f$ is not part of the hypotheses of Theorem \ref{KLseparatelypol} but a nontrivial consequence of $f$ being a separately continuous polynomial map). 
\end{remark}
\begin{remark}
Note that Theorem \ref{KLseparatelypol} can be used when $G=\mathbb{R}^p$, $H=\mathbb{R}^q$ for all $p,q\geq 1$.
\end{remark}
Let $p,q$ be positive integers and $d=p+q$. 
Consider the following quadratic form  
\begin{equation}
\label{qform}
Q(x_1,\cdots,x_p,y_1,\cdots,y_q)=\sum_{k=1}^px_k^2-\sum_{k=1}^qy_k^2
\end{equation}
on $\mathbb{R}^d$ induced by the following bilinear map 
\[
B(z,w)=z^T I_{p,q} w,
\quad\text{where}\quad 
I_{p,q}=\left(\begin{array}{cc}
       I_p & \mathbf{0}\\ 
\mathbf{0} & -I_q \end{array}\right).
\]
A linear map $L:\mathbb{R}^d\to\mathbb{R}^d$ is an isometry with respect to $B$ (equivalently, with respect to $Q$) if 
\[
B(L(z),L(w))=B(z,w),\quad\forall z,w\in\mathbb{R}^d.
\]
The matrices associated with such isometries are elements of the general orthogonal group $\mathbf{O}(p,q)$, and can be characterized as matrices $A$ satisfying $A^TI_{p,q}A=I_{p,q}$. 
\begin{proposition}
\label{O11}
If $d=2$ and $p=q=1$, then $\Lambda_{\mathbf{O}(1,1),(1,0)}$ has a nonempty interior. 
\end{proposition}
\begin{proof}
A direct computation shows that $\mathbf{O}(1,1)$ can be characterized by 
\[
\mathbf{O}(1,1)=\left\{\left(\begin{array}{cc}
a & \varepsilon c\\
c & \varepsilon a
\end{array}\right): (a,c)\in\mathbb{R}^2,\ a^2-c^2=1,\ \varepsilon=\pm 1\right\}.
\]
Let $e_1=(1,0)^T$. 
The orbit $\mathbf{O}(1,1)e_{1}=\{(a,c)^T\in\mathbb{R}^2: a^2-c^2=1\}$ is the hyperbola $x^2-y^2=1$ in $\mathbb{R}^2$.
Since $-I\in \mathbf{O}(1,1)$, the Minkowski sum of the hyperbola with itself describes the set 
\[
\Lambda_{\mathbf{O}(1,1),e_{1}}=\{(a-b,c-d)^T: a^{2}-c^{2}=1,b^{2}-d^{2}=1\},
\] 
which can be shown to have a nonempty interior in $\mathbb{R}^2$. 
Indeed, a formal proof can be obtained by a parametrization $\alpha(u)$ of the hyperbola. 
For example, $\alpha(u)=(\frac{u^2+1}{2u},\frac{u^2-1}{2u})$ does the job. 
We define the map $F:\mathbb{R}^2\to\mathbb{R}^2$ by $F(u,v)=\alpha(u)-\alpha(v)$. 
A direct computation shows that the determinant of the Jacobian matrix 
\[
\det JF(u,v)=\frac{1}{2}\left(\frac{1}{u^2}-\frac{1}{v^2}\right)
\]
is nonzero if and only if $|u|\neq |v|$. 
Thus, the inverse function theorem implies that the image of $F$ has a nonempty interior, and the same holds for
$\Lambda_{\mathbf{O}(1,1),e_1}$.  
\end{proof} 
\begin{theorem}
\label{Opq}
Assume that $f$ belongs to a finite-dimensional space of continuous functions $V\subset C(\mathbb{R}^d)$ such that $V$ is invariant by translations and elements of the general orthogonal group $\mathbf{O}(p,q)$. 
Then $f$ is an ordinary polynomial in $d$ real variables. 
\end{theorem}
\begin{proof}
We separate the proof into different cases:
\par
\textbf{Case 1: $p=q=1$.} 
Proposition \ref{O11} and Theorem \ref{mainT} solve this case.
\par
\textbf{Case 2: $p,q\geq 2$.} 
We claim that for each $b\in\mathbb{R}^q$, $f(x,b)$ belongs to a finite-dimensional space of continuous functions defined on $\mathbb{R}^p$ that is invariant by isometries of  $\mathbb{R}^p$, (thus $f(x,b)$ is an ordinary polynomial, since $p\geq 2$). 
Analogously, the same holds for all functions of the form $f(a,y)$, for each $a\in\mathbb{R}^p$, since $q\geq 2$. 
Indeed, for every pair of matrices $A_1\in \mathbf{O}(p)$ and $A_2\in \mathbf{O}(q)$, it is clear that 
$\left(\begin{array}{cc}
       A_1 & \mathbf{0}\\
\mathbf{0} & I_q
\end{array}\right)$ and 
$\left(\begin{array}{cc}
       I_p & \mathbf{0}\\
\mathbf{0} & A_2
\end{array}\right)$ belong to $\mathbf{O}(p,q)$, 
and for each $a\in \mathbb{R}^p$ and each $b\in \mathbb{R}^q$, if we define operators by the following substitutions: 
\begin{align*}
L_{b}: C(\mathbb{R}^{p}\times\mathbb{R}^{q}) & \to C(\mathbb{R}^{p}), & 
R_{a}: C(\mathbb{R}^{p}\times\mathbb{R}^{q}) & \to C(\mathbb{R}^{q}),\\
f(x,y) & \mapsto L_{b}(f)(x)=f(x,b), & 
f(x,y) & \mapsto R_{a}(f)(y)=f(a,y),
\end{align*}
then we can proceed as follows. 
Take $A\in \mathbf{O}(p)$ and set $\mathbf{A}=\left(\begin{array}{cc}
         A & \mathbf{0} \\
\mathbf{0} & I_q \end{array}\right)\in \mathbf{O}(p,q)$. 
Since each $\phi\in W_b:=L_b(V)$ is of the form $\phi=L_b(g)$ for certain $g\in V$, we have that $O_{\mathbf{A}}(g)\in V$ and 
\[
O_A(\phi)(x)
=O_A(L_b(g))(x)
=g(Ax,b)
=L_b(g(Ax,y))
=L_b(O_{\mathbf{A}}(g))(x)\in W_b.
\]
Analogously, for each $h\in\mathbb{R}^p$, we have $\tau_{(h,\mathbf{0})}(g)\in V$ and 
\[
\tau_h(\phi)(x)
=\tau_h(L_b(g))(x)
=g(x+h,b)
=L_b(g(x+h,y))
=L_b(\tau_{(h,\mathbf{0})}(g))(x)\in W_b.
\]
Thus $W_b$ is finite-dimensional and invariant by translations and elements of $\mathbf{O}(p)$, which, by Theorem \ref{mainT} and item (ii) of Lemma \ref{fatten}, means that all its elements are ordinary polynomials in $p$ real variables, since $p\geq 2$. 
In particular, $f^b(x)=f(x,b)$ is an ordinary polynomial. 
An analogous computation ( using $R_a$, $X_a=R_a(V)$, translations of the form $\tau_{(\mathbf{0},h)}$ with $h\in\mathbb{R}^q$, and the operator $O_{\mathbf{A}}$ with 
$\mathbf{A}=\left(\begin{array}{cc}
       I_p & \mathbf{0}\\
\mathbf{0} & A \end{array}\right)$, $A\in \mathbf{O}(q)$) shows that, for each $a\in\mathbb{R}^p$, $f_a(y)=f(a,y)$ is an ordinary polynomial. 
Thus, $a\in\mathbb{R}^p$ and $b\in\mathbb{R}^q$ imply that $f_a(y)$ and $f^b(x)$ are ordinary polynomials, and Theorem \ref{KLseparatelypol} implies that $f(x,y)$ is an ordinary polynomial in $\mathbb{R}^d$. 
\par
\textbf{Case 3: $p>q=1$.} 
The above argument proves that $f(x_1,\cdots,x_p,b)$ is an ordinary polynomial for each $b\in\mathbb{R}$. 
Moreover, taking $A\in \mathbf{O}(1,1)$ and $P=\left(\begin{array}{cc} 
   I_{p-1} & \mathbf{0}\\ 
\mathbf{0} & A \end{array}\right)$, 
we get $P\in \mathbf{O}(p,q)$ and we can use Theorem \ref{mainT} and Proposition \ref{O11} to claim that $f(a,y_1,y_2)$ is an ordinary polynomial for each $a\in\mathbb{R}^{p-1}$. 
Hence, given $a_1,\cdots,a_p,b\in\mathbb{R}$, we have that both $f(x_1,\cdots,x_p,b)$ and $f(a_1,\cdots,a_p,y)$ are continuous polynomials.  
The result follows from applying Theorem \ref{KLseparatelypol}. 
\par
\textbf{Case 4: $q>p=1$.} 
This case has a symmetric proof compared to the previous one. 
\end{proof} 
\begin{remark}
Note that Cartan-Dieudonné's Theorem \cite[Chapter 1, Theorem 7.1]{L} guarantees that every element of $\mathbf{O}(p,q)$ is a product of at most $d=p+q$ symmetries with respect to non-singular hyperplanes of $\mathbb{R}^d$ (with respect to the quadratic form \eqref{qform}). 
Hence, we only need invariance with respect to these symmetries and translations to characterize the ordinary polynomials defined on $\mathbb{R}^d$.
\end{remark}
\section{The symplectic group $\SP(2d,\mathbb{R})$ case} \label{SSp}
In this section, we consider the symplectic group 
\[
\SP(2d,\mathbb{R})=\{M\in\GL(2d,\mathbb{R}): M^T\Omega_d M=\Omega_d\},
\quad\text{where}\quad
\Omega_d=\left(\begin{array}{cc} 
\mathbf{0} & I_d\\ 
      -I_d & \mathbf{0} \end{array}\right),
\]
which is also a group of isometries. 
Indeed, if we define on $\mathbb{R}^{2d}$ the following bilinear form 
\begin{equation}
\label{bform}
\mathcal{B}(x,y)=y^T\Omega_dx,
\end{equation}
then the elements of $\SP(2d,\mathbb{R})$ are the linear isometries of $\mathbb{R}^{2d}$ associated to $\mathcal{B}$ 
(i.e. $M\in \SP(2d,\mathbb{R})$ if and only if $\mathcal{B}(Mx,My)=\mathcal{B}(x,y)$ for all $x,y\in\mathbb{R}^{2d}$).  
\begin{proposition}
\label{SL1}
$\Lambda_{\SL(2,\mathbb{R}),(1,0)}=\Lambda_{\SP(2,\mathbb{R}),(1,0)}$ has a nonempty interior. 
\end{proposition}
\begin{proof}
This follows from item (ii) of Lemma \ref{fatten} or from Example \ref{SO2}, since $\mathbf{O}(2)\subseteq \SP(2,\mathbb{R})=\SL(2,\mathbb{R})=\{A\in \GL(2,\mathbb{R}):\det(A)=1\}$.
\end{proof} 
Now we can prove the main result of this section:
\begin{theorem} \label{Sp}
Assume that $f$ belongs to a finite-dimensional subspace $V\subset C(\mathbb{R}^{d}\times \mathbb{R}^{d})$ of continuous functions such that $V$ is invariant by translations and elements of the symplectic group $\SP(2d,\mathbb{R})$. 
Then $f$ is an ordinary polynomial in $2d$ real variables. 
\end{theorem}
\begin{proof}
As before, for each $a,b\in\mathbb{R}^d$, consider the operators $L_b,R_a:C(\mathbb{R}^{d}\times \mathbb{R}^{d})\to C(\mathbb{R}^{d})$ given by $L_b(f)(x)=f(x,b)$ and $R_a(f)(y)=f(a,y)$. 
We claim that $W_{\mathbf{0}}=L_{\mathbf{0}}(V)$ is formed by ordinary polynomials. 
Indeed, take $\phi\in W_{\mathbf{0}}$. 
Then $\phi(x)=L_{\mathbf{0}}(g)(x)=g(x,{\mathbf{0}})$ for certain $g(x,y)\in V$. 
To compute $\tau_h(\phi)$ with $h\in \mathbb{R}^d$, we have that
\[
\tau_h(\phi)(x)=\tau_h(L_{\mathbf{0}}(g))(x)=g(x+h,{\mathbf{0}})=L_{\mathbf{0}}(g(x+h,y))=L_{\mathbf{0}}(\tau_{(h,{\mathbf{0}})}(g))(x)\in W_{\mathbf{0}}.
\]
To compute $O_A(\phi)$ with $A\in \GL(d,\mathbb{R})$, we consider the matrix 
\[
\mathcal{A}=\left(\begin{array}{cc} 
         A & \mathbf{0}\\ 
\mathbf{0} & (A^T)^{-1} \end{array}\right)\in\SP(2d,\mathbb{R}),
\]
and have that 
\[
O_A(\phi)(x)
=O_A(L_{\mathbf{0}}(g))(x)
=g(Ax,\mathbf{0})
=L_{\mathbf{0}}(g(Ax,(A^T)^{-1}y))
=L_{\mathbf{0}}(O_{\mathcal{A}}(g))(x)\in W_{\mathbf{0}}.
\]
Thus $W_{\mathbf{0}}$ is invariant by translations and elements of $\GL(d,\mathbb{R})$, and thanks to Theorem \ref{mainT} and item (iv) of Lemma \ref{lemma_general}, all its elements are ordinary polynomials. 
Given $b\in\mathbb{R}^d$ and $g\in V$, 
\[
g^b(x)
:=g(x,b)
=\tau_{(\mathbf{0},b)}(g)(x,\mathbf{0})
=L_{\mathbf{0}}(\tau_{(\mathbf{0},b)}(g))\in L_{\mathbf{0}}(V)=W_{\mathbf{0}},
\]
which implies that $g^b(x)$ is also an ordinary polynomial. 
A similar computation shows that for any $g\in V$ and $a\in\mathbb{R}^d$, $g_a(y):=g(a,y)$ is an ordinary polynomial on $\mathbb{R}^d$. 
Again, a direct application of Theorem \ref{KLseparatelypol} shows that all elements of $V$ are ordinary polynomials.
\end{proof}
\section{A general remark about the distributional case} 
\label{Dist}
It is remarkable that we have been able to characterize ordinary polynomials by their invariance properties with respect to several classical groups $G\subseteq \GL(d,\mathbb{R}^d)$, avoiding the use of exponential polynomials. 
In Sections \ref{SOpq} and \ref{SSp}, we need to impose the continuity of the involved functions in our arguments. 
This can be avoided -- so that the characterizations given in Theorems \ref{Opq} and \ref{Sp} also hold for distributions -- if we use Anselone-Korevaar's theorem \cite{AK}, 
stating that quasipolynomial distributions (i.e. distributions $f\in \mathcal{D}(\mathbb{R}^d)'$ that belong to a finite-dimensional translation invariant subspace of $\mathcal{D}(\mathbb{R}^d)'$ ) are exponential polynomials. 
\begin{theorem}
Let $G$ be a subgroup of $\GL(d,\mathbb{R})$ and let $\mathcal{V}$ be a vector subspace of $C(\mathbb{R}^d)$. 
Assume that $\mathcal{V}$ has the following property:
\begin{itemize}
\item[(P)] A function $f\in C(\mathbb{R}^d)$ belongs to $\mathcal{V}$ if and only if there exists a finite-dimensional vector space $V\subset C(\mathbb{R}^d)$ such that: 
(i) $f\in V$ and 
(ii) $V$ is invariant by translations and elements of $G$.
\end{itemize}
Then a distribution $f\in\mathcal{D}(\mathbb{R}^d)'$ belongs to $\mathcal{V}$ if and only if there exists a finite-dimensional vector space $V\subset X_d$ such that: 
(i) $f\in V$ and 
(ii) $V$ is invariant by translations and elements of $G$.
\end{theorem}
\begin{proof}
This is a direct consequence of Anselone-Korevaar's Theorem. 
Indeed, if $f\in \mathcal{D}(\mathbb{R}^d)'$ belongs to $V\subset X_d$, which is finite-dimensional, invariant by translations and elements of $G$, then $f$ is an exponential polynomial in the distributional sense. 
In particular, $f$ can be assumed to be continuous, since it is equal to a continuous function almost everywhere. 
Thus $R_G^c(f)$, the smallest subspace of $C(\mathbb{R}^d)$ that contains $f$ and is invariant by the operators $O_P$ and $\tau_h$ for all $h\in\mathbb{R}^d$ and $P\in G$, is a subspace of $V$, has finite dimension, and is translation invariant and invariant by elements of $G$. 
Hence $(P)$ implies that $f\in\mathcal{V}$. 
\end{proof}
\section{Polynomials defined on $\mathbb{K}^d$}
Consider scalar functions $f:\mathbb{K}^d\to\mathbb{K}$, where $\mathbb{K}$ is a field of characteristic zero. 
A special class of polynomial maps $f:\mathbb{K}^d\to\mathbb{K}$ is formed by those functions which belong to the function algebra generated by the additive functions $a:\mathbb{K}^d\to\mathbb{K}$ (i.e., $a(x+y)=a(x)+a(y), \forall x,y\in\mathbb{K}^d$), and the constants. 
These functions are simply called {\it polynomials} or, sometimes, {\it normal polynomials}. 
\par
In the case $\mathbb{K}=\mathbb{R}$ or $\mathbb{K}=\mathbb{C}$, every continuous polynomial map is a polynomial. 
In fact, it is an ordinary polynomial. 
Moreover, the following theorem is known (see e.g. \cite[Theorem 8]{Schwaiger_Prager_2008}).
\begin{theorem}
\label{char1}
Let $f:\mathbb{K}^d\to\mathbb{K}$ be a polynomial map, where $\mathbb{K}$ is a field of characteristic zero. 
Then $f$ is a polynomial if and only if the dimension of $\rm{span} \{\tau_h(f):h\in\mathbb{K}^d\}$ is finite.
\end{theorem}
Now, if the function $f$ is of the form $f:\mathbb{K}^d\to\mathbb{K}$ for a certain field $\mathbb{K}$ of characteristic $0$ and $q(\tau_h)(f)=0$ for all $h\in\mathbb{K}^d$ and certain ordinary polynomial $q(x)=a_0+a_1x+\cdots+a_nx^n\in\mathbb{K}[x]$, the arguments used for the proof of Lemma \ref{pre} show that $f$ is a polynomial function on $\mathbb{K}$ of degree $\leq n-1$. 
Moreover, if $\rm{span} \{\tau_h(f):h\in\mathbb{K}^d\}$ is finite-dimensional, then Theorem \ref{char1} implies that $f=P(a_1(x),\cdots,a_s(x))$ for a certain polynomial $P\in\mathbb{K}[x_1,\cdots,x_s]$ of total degree $\leq n-1$ and certain additive functions $a_i:\mathbb{K}^d\to\mathbb{K}$. 
Now, a simple computation shows that, for any field $\mathbb{K}$ and any $d\geq 1$, $\Lambda_{\GL(d,\mathbb{K}),z_0}=\mathbb{K}^d$ for each $z_0\in\mathbb{K}^d\setminus\{0\}$. 
Indeed, given $z_1\in \mathbb{K}^d\setminus \{0\}$, there exists $P\in \GL(d,\mathbb{K})$ such that $Pz_0=z_1$. 
Thus $\GL(d,\mathbb{K})z_0=\mathbb{K}^d\setminus\{0\}$, which implies that $\Lambda_{\GL(d,\mathbb{K}),z_0}=\mathbb{K}^d$,  and a similar argument to the one given at Theorem \ref{mainT}  shows the following:
\begin{theorem}
\label{general_K}
Let $\mathbb{K}$ be a field of characteristic $0$. 
If $f:\mathbb{K}^d\to\mathbb{K}$ belongs to a finite-dimensional space of $\mathbb{K}$-valued functions defined on $\mathbb{K}^d$ that is invariant by translations and elements of $\GL(d,\mathbb{K})$, then $f$ is a polynomial.
\end{theorem}
For the proof of Theorem \ref{general_K}, we can restrict our attention to dilations $D_k\in \GL(d,\mathbb{K})$, $D_k(x_1,\cdots,x_d)=(kx_1,\cdots,kx_d)$, $k\in\mathbb{N}=\{1,2,\cdots\}$, and avoid using Lemma \ref{pre} as well as the arguments from the proof of Theorem \ref{mainT}. 
Instead, we use the following technical result:
\begin{lemma}
\label{raices}
Let $\mathbb{K}$ be a field and $A\subset\mathbb{K}\setminus\{0\}$ be a nonempty finite set such that 
$A=\varphi_n(A)$ for infinitely many positive integers $n$, where $\varphi_n(z)=z^n$ for all $z\in A$. 
Then $A=\{1\}$. 
\end{lemma}
\begin{proof}
Obviously, $A=\{1\}$ satisfies all the given conditions. 
Assume that $A$ contains an element $z$ other than $1$. 
Let $S$ be the infinite subset of $\mathbb{N}$ such that $A=\varphi_n(A)$ for all $n\in S$. 
As $A$ is finite, $S$ must contain two integers $m,n$ such that $z^m=z^n\in A$. 
We may assume that $k=m-n>0$. 
Then $z^n(z^k-1)=0$ and $z^n\neq 0$ imply that $\varphi_k(z)=z^k=1\in A$. 
Since $A$ is finite, $\varphi_k(A)=A$ implies that $\varphi_k$ is bijective. 
This is impossible since $z\neq 1$ but $z^k=1=1^k$.
\end{proof}
\begin{theorem}
\label{general_K_dilations}
Let $\mathbb{K}$ be a field of characteristic $0$. 
If $f:\mathbb{K}^d\to\mathbb{K}$  belongs to a finite-dimensional space $V$ of $\mathbb{K}$-valued functions defined on $\mathbb{K}^d$ that is invariant by translations and dilations $D_k$ for infinitely many $k\in\mathbb{N}$, then $f$ is a polynomial.
\end{theorem}
\begin{proof}
Indeed, given $h\in\mathbb{K}^d$, we have $(\tau_h)^k=\tau_{kh}=\tau_{D_kh}=(O_{D_k})^{-1}\tau_{h}O_{D_k}$, 
which means that $\tau_h$ and $(\tau_h)^k$ share the same spectrum for infinitely many $k\in\mathbb{N}$. 
Let $S=\{\lambda_1,\cdots,\lambda_s\}$ be the spectrum of $\tau_h:V\to V$ (eliminating multiplicities, so that the eigenvalues are pairwise distinct).
Note that $\tau_h(f)=\lambda f$ implies $(\tau_h)^k(f)=\lambda^k f$. 
We have that $S=\{\lambda_1,\cdots,\lambda_s\}=\{\lambda_1^k,\cdots,\lambda_s^k\}$ for infinitely many $k\in\mathbb{N}$.
Since $\tau_h$ is an isomorphism for all $h$, the eigenvalues $\lambda_i$ are all nonzero. 
Lemma \ref{raices} implies that $S=\{1\}$, which means that $(z-1)^n$ is the characteristic polynomial of $\tau_h$.
Now, Hamilton's theorem implies that $\Delta_h^nf=(\tau_h-1)^nf=0$ for all $h$, where $n=\dim V$. 
Thus, $f$ is a polynomial function of functional degree at most $n-1$.
At last, Theorem \ref{char1} implies that $f=P(a_1(x),\cdots,a_s(x))$ for a certain polynomial $P\in\mathbb{K}[x_1,\cdots,x_s]$ of total degree $\leq n-1$ and certain additive functions $a_i:\mathbb{K}^d\to\mathbb{K}$.    
\end{proof}
\begin{remark}
Note that Theorem \ref{general_K_dilations} also applies for $\mathbb{K}=\mathbb{R}$, and that, for $d\geq 2$, the result cannot be proved as a corollary of Theorem \ref{mainT}, 
because the group of dilations $D=\{aI_d:a\in\mathbb{R}\setminus\{0\}\}$ satisfies $\Lambda_{D,z_0}=\mathbf{span}\{z_0\}$ for all $z_0\in\mathbb{R}^d$, which has an empty interior in these cases. 
\end{remark}
\begin{remark}
When $\mathbb{K}$ has positive characteristic, we cannot use Theorem \ref{char1}. 
Thus, it would be great to know if Theorem \ref{char1} can also be demonstrated for fields of positive characteristic. 
In such a case, Theorems \ref{general_K} and \ref{general_K_dilations} would also hold for fields of positive characteristic. 
Meanwhile, if ${\rm char}(\mathbb{K})=p>0$, and $f:\mathbb{K}^d\to\mathbb{K}$ satisfies the other hypotheses of either Theorem \ref{general_K} or Theorem \ref{general_K_dilations}, we can only claim that $f$ is a polynomial map. 
\end{remark}
%
%
%
%
%

\end{document}